\documentclass[11pt]{amsart}

\usepackage{amsmath}
\usepackage{fullpage}
\usepackage{xspace}
\usepackage[psamsfonts]{amssymb}
\usepackage[latin1]{inputenc}
\usepackage{graphicx,color}
\usepackage[curve]{xypic} 
\usepackage{hyperref}
\usepackage{graphicx}

\usepackage{amsmath}%
\usepackage{amsthm}%
\usepackage{amscd}
\usepackage{amsfonts}%
\usepackage{amssymb}%
\usepackage{graphicx}

\usepackage{mathrsfs}

\usepackage{tikz}
\usetikzlibrary{matrix,arrows}

\newtheorem{theorem}{Theorem}[section]

\newtheorem{corollary}[theorem]{Corollary}

\newtheorem{lemma}[theorem]{Lemma}

\theoremstyle{remark}

\numberwithin{equation}{section}

\newcommand{\Fcal}{\mathscr{F}}

\newcommand{\Lcal}{\mathscr{L}}
\newcommand{\Mcal}{\mathscr{M}}

\newcommand{\Ocal}{\mathscr{O}}

\newcommand{\Pro}{\mathbb{P}}

\newcommand{\Z}{\mathbb{Z}}
\newcommand{\C}{\mathbb{C}}

\newcommand{\F}{\mathbb{F}}
\newcommand{\Q}{\mathbb{Q}}
\newcommand{\R}{\mathbb{R}}

\newcommand{\A}{\mathbb{A}}

\newcommand{\E}{\mathbb{E}}

\newcommand{\ord}{\mathrm{ord}}

\newcommand{\rk}{\mathrm{rank}\,}
\newcommand{\supp}{\mathrm{supp}\,}

  \DeclareFontFamily{U}{wncy}{}
    \DeclareFontShape{U}{wncy}{m}{n}{<->wncyr10}{}
    \DeclareSymbolFont{mcy}{U}{wncy}{m}{n}
    \DeclareMathSymbol{\Sha}{\mathord}{mcy}{"58}

\begin{document}
\title[]{Elliptic curves with long arithmetic progressions have large rank}

\author{Natalia Garcia-Fritz}
\address{ Departamento de Matem\'aticas\newline
\indent Pontificia Universidad Cat\'olica de Chile\newline
\indent Facultad de Matem\'aticas\newline
\indent 4860 Av. Vicu\~na Mackenna\newline
\indent Macul, RM, Chile}
\email[N. Garcia-Fritz]{natalia.garcia@mat.uc.cl}

\author{Hector Pasten}
\address{ Departamento de Matem\'aticas\newline
\indent Pontificia Universidad Cat\'olica de Chile\newline
\indent Facultad de Matem\'aticas\newline
\indent 4860 Av. Vicu\~na Mackenna\newline
\indent Macul, RM, Chile}
\email[H. Pasten]{hector.pasten@mat.uc.cl}%

\thanks{N. G.-F. was supported by the FONDECYT Iniciaci\'on en Investigaci\'on grant 11170192 and the CONICYT PAI grant 79170039. H.P. was supported by FONDECYT Regular grant 1190442.}
\date{\today}

\begin{abstract} For any family of elliptic curves over the rational numbers with fixed $j$-invariant, we prove that the existence of a long sequence of rational points whose $x$-coordinates form a non-trivial arithmetic progression implies that the Mordell-Weil rank is large, and similarly for $y$-coordinates. We give applications related to uniform boundedness of ranks, conjectures by Bremner and Mohanty, and arithmetic statistics on elliptic curves. Our approach involves  Nevanlinna theory as well as R\'emond's quantitative extension of results of Faltings.
\end{abstract}

\maketitle


\section{Introduction}

 It is an open problem whether the ranks of elliptic curves over $\Q$ are uniformly bounded. Various heuristics have been developed in support of uniform boundedness \cite{ BhKLPR, PaPVW, PoICM, PoR, RubinSi, WaDEFGR}. Also, the second author has shown  \cite{PastenLangRks} that a conjecture of Lang in diophantine approximation implies uniform boundedness of ranks for families of elliptic curves with a fixed $j$-invariant. In the direction of unboundedness, it is known that elliptic curves over a global function field such as $\F_p(t)$ can have arbitrarily large rank even if one considers quadratic twists families \cite{ShaTat, Ulm} and examples over $\Q$ with remarkably large rank are known \cite{ElkiesRks}. It is natural to look for a mechanism forcing the rank of elliptic curves over $\Q$ to be large, and certain patterns on rational points seem to achieve this.

Given an elliptic curve $E$ over $\Q$, an \emph{$x$-arithmetic progression} is a sequence $P_1,...,P_N$ of $\Q$-rational points on $E$ having their $x$-coordinates in arithmetic progression for some choice of Weierstrass equation $y^2=x^3+ax^2+bx+c$ for $E$. A \emph{$y$-arithmetic progression} on $E$ is defined similarly. Such sequences are said to be \emph{non-trivial} if the resulting arithmetic progression in $x$ or $y$ coordinates is non-constant. These definitions are in fact independent of the choice of Weierstrass equation.

 Bremner \cite{Bremner} has conjectured that rational points of an $x$-arithmetic progression on an elliptic curve $E$ over $\Q$ tend to be linearly independent in the Mordell-Weil group. The conjecture is motivated by numerous examples, as well as theoretical evidence such as \cite{BST} where it is shown that for a quadratic twist family over $\Q$, the elliptic curves of rank $1$ have $x$-arithmetic progressions of uniformly bounded length. See also \cite{Campbell, GSTor, MoZa, Spearman, Ulas} and the references therein for further examples supporting Bremner's conjecture.

In this work we prove Bremner's conjecture for families of elliptic curves with fixed $j$-invariant, in particular, for quadratic twist families. Given an elliptic curve $E$ over $\Q$, we let $\beta_x(E)$ be the maximal length of a non-trivial $x$-arithmetic progression of rational points in $E$, and $\beta_y(E)$ is defined similarly for $y$-arithmetic progressions. We prove
\begin{theorem}\label{ThmIntro} Let $j_0\in \Q$. There is an effectively computable constant $c(j_0)>0$ that only depends on $j_0$, such that for every elliptic curve $E$ over $\Q$ with $j$-invariant equal to $j_0$ we have 
$$
1+\rk E(\Q)\ge c(j_0)\cdot \log \max\{\beta_x(E),\beta_y(E)\}.
$$
\end{theorem}
In this work, by ``effectively computable'' we always mean that an explicit closed formula can be obtained after some calculation. For instance, in Theorem \ref{ThmIntro} we can take
\begin{equation}\label{Eqncj0}
c(j_0)=\frac{1}{80^{46}\left(209 +  \max\{  \log (1+\log H(j_0)) , 255 + \log\left( 2 + \log(1+\log H(j_0))\right)\}\right)  }
\end{equation}
where $H(j_0)$ is the naive height, namely $H(a/b)=\max\{|a|,|b|\}$ for coprime integers $a,b$ with $b\ne 0$. Although it should be clear from the formula that we did not try to numerically optimize this estimate, let us remark that this value for $c(j_0)$ is satisfactory in the aspect that it is of the order of magnitude of $1/(\log \log H(j_0))$, so, it tends to $0$ extremely slowly. 

We prove that \eqref{Eqncj0} is admissible for Theorem \ref{ThmIntro} in Section \ref{SecPazuki} using (among other tools) a comparison between the Theta height and the Faltings height of abelian varieties  \cite{Pazuki}.

It turns out that arithmetic progressions on elliptic curves not only relate to the rank. We also prove that the (algebraic) torsion points do not have arbitrarily long patterns of this type.  
\begin{theorem} \label{ThmTorIntro} Let $E$ be an elliptic curve over $\Q^{alg}$ with a given Weierstrass equation.  The set of $x$-coordinates of the torsion points of $E(\Q^{alg})$ does not contain arbitrarily long non-trivial  arithmetic progressions. The same holds for $y$-coordinates. A bound for the length of such progressions can be effectively computed from the $j$-invariant of $E$.
\end{theorem}
Heuristically, these results are consistent with the fact that the group structure on elliptic curves  is incompatible with the  additive structure of the affine line via the $x$ or $y$-coordinate maps.

 Theorems \ref{ThmIntro} and \ref{ThmTorIntro} (cf.\ Section \ref{SecConsequences}) are special cases of Theorem \ref{ThmMainArith} which concerns elliptic curves over number fields and more general patterns on algebraic points, not just arithmetic progressions on $x$ or $y$-coordinates of rational points.  Our proof of Theorem \ref{ThmMainArith} heavily uses Nevanlinna's value distribution theory for complex holomorphic maps in order to compute the Kawamata locus of certain sub-varieties of abelian varieties. This will allow us to apply Remond's quantitative version of Faltings' theorem on rational points of sub-varieties of abelian varieties (cf.\ \cite{Faltings2,Faltings3,Remond1}), which will be our main tool to control  $x$ and $y$-arithmetic progressions on elliptic curves.

Let us now discuss some applications of Theorem \ref{ThmIntro}.
Given an elliptic curve $E$ over $\Q$ and a squarefree integer $D$, we let $E^{(D)}$ be the quadratic twist of $E$ by $D$. The number of distinct prime factors of $D$ is denoted by $\omega(D)$. From Theorem \ref{ThmIntro} and standard rank bounds we deduce
\begin{corollary}[cf.\ Sec.\ \ref{SecPtRkBd}]\label{CoroTwistIntro} Given an elliptic curve $E$ over $\Q$, there is an effectively computable constant $C(E)>0$ such that for every squarefree integer $D$ we have
$$
\max\{\beta_x(E^{(D)}),\beta_y(E^{(D)})\} \le C(E)^{\omega(D)+1}.
$$
\end{corollary}
In connection with conjectures on uniform boundedness of ranks, Theorem \ref{ThmIntro} directly gives
\begin{corollary} \label{CoroBddRkMohanty} Let $j_0\in\Q$. Suppose that the elliptic curves over $\Q$ of $j$-invariant equal to $j_0$ have uniformly bounded Mordell-Weil rank. Then there is a number $N(j_0)$ only depending on $j_0$ with the following property: For each elliptic curve $E$ over $\Q$ with $j$-invariant equal to $j_0$ we have $$
\max\{\beta_x(E),\beta_y(E)\} \le N(j_0).
$$
\end{corollary}
We remark that, in view of \cite{PastenLangRks}, the assumption that elliptic curves over $\Q$ with a fixed $j$-invariant have uniformly bounded Mordell-Weil rank, is implied by a conjecture of Lang on the error terms in Diophantine approximation.

 Mohanty \cite{Mohanty1, Mohanty2} conjectured that there is a uniform bound for the length of $x$ and $y$-arithmetic progressions on Mordell elliptic curves $A_n: y^2=x^3+n$ with $n\in \Z$ sixth-power free. Mohanty in fact made the stronger conjecture that $\beta_x(A_n)$ and $\beta_y(A_n)$ are at most $4$, but this was disproved for $y$-arithmetic progressions by Lee and Velez \cite{LeeVelez}. Several constructions as well as extensive numerical searches have been carried out looking for long $x$ or $y$-arithmetic progressions on Mordell elliptic curves (cf.\ \cite{Ulas} and the references therein), but the record continues to be $x$-arithmetic progressions of length $4$ and $y$-arithmetic progressions of length $6$ as found by Lee and Velez \cite{LeeVelez}. 
  
Mohanty's conjecture on uniform boundedness of $x$ and $y$ arithmetic progressions on Mordell elliptic curves remains open. In support of this conjecture besides the search for examples, the first author used extensions of methods by Bogomolov and Vojta to show that the case of $y$-arithmetic progressions follows from the Bombieri-Lang conjecture for surfaces of general type \cite{GF}.  In addition, let us remark that Corollary \ref{CoroBddRkMohanty} with $j_0=0$ gives
\begin{corollary} The uniform boundedness conjecture for ranks of elliptic curves over $\Q$ with $j$-invariant equal to $0$ implies Mohanty's conjecture for both $x$ and $y$-arithmetic progressions.
\end{corollary}
Theorem \ref{ThmIntro} also allows us to prove \emph{unconditionally} that Mohanty's conjecture holds on average, in the sense that the average  $\tau$-moments of $\beta_x(A_n)$ and $\beta_y(A_n)$ are finite for certain $\tau>0$.
\begin{theorem}[cf.\ Sec.\ \ref{SecAvgMordell}] \label{ThmAvgMohanty} For $x>0$ let $S(x)$ be the set of sixth-power free integers $n$ with $|n|\le x$. There are absolute constants $\tau,M>0$ such that for all $x>1$ we have
$$
\frac{1}{\# S(x)}\sum_{n\in S(x)} \max\{\beta_x(A_n),\beta_y(A_n)\}^\tau <M.
$$
\end{theorem}
The proof of Theorem \ref{ThmAvgMohanty} combines Theorem \ref{ThmIntro} with results of Fouvry \cite{Fouvry} on upper bounds for the average size of $3$-isogeny Selmer groups for Mordell elliptic curves, which in turn relies on the Davenport-Heilbronn theorem on $3$-torsion of class groups of quadratic fields. See \cite{BES} for the exact computation of the average size of the $3$-isogeny Selmer groups of Mordell elliptic curves.

More generally, Theorem \ref{ThmIntro} allows us to study arithmetic statistic questions related to $\beta_x(E)$ and $\beta_y(E)$ provided that we have good control on Selmer groups. Indeed, given a positive integer $m$ and an elliptic curve $E$ over $\Q$ we have the exact sequence
\begin{equation}\label{EqnExact}
0\to E(\Q)/mE(\Q)\to S_m(E)\to \Sha(E)[m]\to 0
\end{equation}
where $S_m(E)$ is the $m$-Selmer group and $\Sha(E)[m]$ is the $m$-torsion of the Shafarevich-Tate group. The classical proof of the Mordell-Weil theorem shows that $S_m(E)$ is finite, and from \eqref{EqnExact} we have
$$
m^{\rk E(\Q)}\le \#\left(E(\Q)/mE(\Q)\right)\le \# S_m(E).
$$
Therefore, estimates for the size of $m$-Selmer groups give bounds for \emph{exponential} functions of the rank, and we remark that there are several strong results in the literature for the arithmetic statistics for the size of $m$-Selmer group of elliptic curves. This  is well-suited for our applications, as Theorem \ref{ThmIntro} (and, more generally, Theorem \ref{ThmMainArith}) bounds the maximal length of an arithmetic progression in terms of an exponential function of the rank. For applications along these lines, it is crucial that our lower bound for the rank in Theorem \ref{ThmIntro} is logarithmic; see Section \ref{SecApplications} for details.

As the literature on arithmetic statistics of $m$-Selmer groups of elliptic curves is abundant and growing, we will only focus on the particularly convenient case of the  elliptic curves $B_{n}$ defined by $y^2=x^3-n^2x$. These elliptic curves are associated to the classical ``congruent number problem''. Here, results of Heath-Brown \cite{HB} allow us to control \emph{all} the average moments of $\beta_x(B_n)$ and $\beta_y(B_n)$.
\begin{theorem}[cf.\ Sec.\ \ref{SecAvgCongruent}]\label{ThmAvgCongruent} Let $Q(x)$ be the set of odd squarefree integers $n$ with $1\le n\le x$. Let $k >0$. There is a constant $M(k)>0$ depending only on $k$ such that for all $x>1$ we have
$$
\frac{1}{\#Q(x)}\sum_{n\in Q(x)} \max\{\beta_x(B_n),\beta_y(B_n)\}^k <M(k).
$$
\end{theorem}
Let us mention that $x$-arithmetic progressions on the elliptic curves $B_n$ are studied in detail in \cite{BST} \emph{under the assumption $\rk B_n(\Q)\le 1$} and in \cite{Spearman} for a specific sub-family arising from an elliptic surface of rank $3$. The study in \cite{BST} is motivated by a connection with the problem of existence of a $3\times 3$ magic square formed by different integer squares \cite{Robertson}.


\section{Review of Nevanlinna theory}

In this section we set up the notation regarding Nevanlinna theory for holomorphic maps into complex projective varieties. Specially, we introduce the counting, proximity and height functions. We also recall the fundamental properties of these functions, including the First and Second Main Theorems. All the results in this section are standard and we include them for later reference. See for instance \cite{VojtaCIME} for proofs and more general versions of the results in this section.

\subsection{Definitions} Let $X$ be a smooth projective variety over $\C$ (we will identify the algebraic variety $X$ with the complex manifold $X(\C)$ if no confusion can arise).   Let $D$ be a divisor on $X$ and for each point $x\in X$ we let $\phi_{D,x}$ be a local equation for $D$. The support of $D$ is $\supp D$. Given a complex holomorphic map $f:\C\to X$ with image not contained in $\supp D$, we define for $r\ge 0$ 
$$
n_X(f, D, r)=\sum_{|z|\le r} \ord_z(f^*\phi_{D,f(z)}).
$$
For each $r$ the sum is finite. The \emph{counting function} is
$$
\begin{aligned}
N_X(f,D,r)&=\int_{0}^r\left(n_X(f, D, t)-n_X(f, D, 0)\right) \frac{dt}{t} + n(f, D, 0)\log r\\
&= \sum_{0<|z|\le r} \ord_z(f^*\phi_{D,f(z)})\log \frac{r}{|z|} +  \ord_0(f^*\phi_{D,f(z)})\log r.
\end{aligned}
$$
If $f(0)\notin \supp D$, then the counting function takes the simpler form
$$
N_X(f,D,r)=\int_{0}^r n(f, D, t) \frac{dt}{t}.
$$
A \emph{Weil function for $D$} is a function $\lambda_{X,D}:X-\supp D\to \R$ satisfying that for each $x\in X$ there is a complex neighborhood $U_x\subseteq X$ of $x$ and a continuous function $\alpha_x:U_x\to \R$ such that $\lambda_{X,D}(y)= - \log |\phi_{D,x}(y)| + \alpha_x(y)$ for all $y\in U_x-\supp D$. It is a standard result that Weil functions for $D$ exist, and they are unique up to a bounded continuous function on $X-\supp D$. 

With $f:\C\to X$ and $D$ as before and a choice of Weil function $\lambda_{X,D}$, the \emph{proximity function} is
$$
m_X(f,D,r)=\int_{0}^{2\pi } \lambda_{X,D}\left(f(r\cdot \exp( i \theta))\right)\frac{d\theta }{2\pi} .
$$
 The function $m_X(f,D,-):\R_{\ge 0}\to \R$ is well-defined up to adding a bounded function.

The \emph{Nevanlinna height} of $f$ with respect to $D$ is the function $T_X(f,D,-):\R_{\ge 0}\to \R$ defined by
$$
T_X(f,D,r)=N_X(f,D,r)+m_X(f,D,r).
$$
Due to the choice of Weil function in $m_X(f,D,-)$, we have that $T_X(f,D,-)$ is well defined up to a bounded function of $r$. 

\subsection{Basic properties} Let us briefly recall some of the fundamental properties of the counting, proximity, and height functions. We use Landau's notation $u(x)=O(v(x))$ for functions $u,v:\R_{\ge 0}\to \C$ with $v$ positive valued, to indicate that there is a constant $M$ independent of $x$ such that for all $x>0$ we have $|u(x)|\le M\cdot v(x)$.
\begin{lemma} Let $X$ be a smooth complex projective variety and $f:\C\to X$  a holomorphic map. 
\begin{itemize}
\item (Additivity) Let $D_1,D_2$ be divisors on $X$ such that the image of $f$ is not contained in $\supp D_1\cup \supp D_2$. Let $a,b\in \Z$. Then for all $r>0$ we have
$$
\begin{aligned}
N_X(f,aD_1+bD_2,r)&=aN_X(f,D_1,r)+bN_X(f,D_2,r)\\
m_X(f,aD_1+bD_2,r)&=am_X(f,D_1,r)+bm_X(f,D_2,r) + O(1)\\
T_X(f,aD_1+bD_2,r)&=aT_X(f,D_1,r)+bT_X(f,D_2,r) + O(1)
\end{aligned}
$$
where the error terms are independent of $r$.
\item (Effectivity) Let $D$ be an effective divisor on $X$ such that the image of $f$ is not contained in $\supp D$. Then for all $r\ge 1$ we have
$$
N_X(f,D,r)\ge 0,\quad 
m_X(f,D,r)\ge O(1),\quad
T_X(f,D,r)\ge O(1)
$$
where the error terms are independent of $r$.
\item (Functoriality) Let $Y$ be a smooth complex projective variety, $D$ a divisor on $Y$ and let $\gamma : X\to Y$ be a morphism. If the image of $\gamma\circ f$ is not contained in $\supp D$, then 
$$
\begin{aligned}
N_X(f,\gamma^*D,r)&=N_Y(\gamma \circ f,D,r)\\
m_X(f,\gamma^*D,r)&=m_Y(\gamma \circ f,D,r) + O(1)\\
T_X(f,\gamma^*D,r)&=T_Y(\gamma \circ f,D,r) + O(1)
\end{aligned}
$$
where the error terms are independent of $r$. 
\end{itemize}
\end{lemma}
\begin{lemma}[Ample height property] Let $X$ be a smooth complex projective variety. Let $f:\C\to X$ be a holomorphic map.  Let $D$ be an ample divisor on $X$ such that the image of $f$ is not contained in $\supp D$. If $f$ is non-constant, then $T_X(f,D,r)$ grows to infinity. 
\end{lemma}
\begin{lemma}[First Main Theorem] Let $X$ be a smooth complex projective variety and let $f:\C\to X$ be a holomorphic map. Let $D_1,D_2$ be linearly equivalent divisors on $X$ such that the image of $f$ is not contained in $\supp D_1\cup \supp D_2$. Then $T_X(f,D_1,r)= T_X(f,D_2,r) + O(1)$.
\end{lemma}

\subsection{Truncated counting functions} When $D$ is an effective reduced divisor on $X$ and the image of $f:\C\to X$ is not contained in $\supp D$, we define
$$
n^{(1)}_X(f,D,r)=\# \{z\in \C : |z|\le r\mbox{ and }f(z)\in \supp D\}
$$
and the \emph{truncated counting function}
$$
N^{(1)}_X(f,D,r)=\int_{0}^r\left(n^{(1)}_X(f, D, t)-n^{(1)}_X(f, D, 0)\right) \frac{dt}{t} + n^{(1)}_X(f, D, 0)\log r.
$$
We note that $N^{(1)}_X(f,D,r)\ge 0$ for $r\ge 1$. In general, the truncated counting function does not respect additivity. It is useful to observe that for an effective reduced divisor $D$ on $X$ and a holomorphic map $f:\C\to X$ whose image is not contained in $\supp D$, for all $r\ge 1$ we have
\begin{equation}\label{EqnN1NT}
0\le N^{(1)}_X(f,D,r)\le N_X(f,D,r)\le T_X(f,D,r) +O(1)
\end{equation}
where the last estimate is due to the effectivity property of $m_X(f,D,r)$.

\subsection{Second Main Theorem} Let us state the Second Main Theorem of Nevanlinna theory in the case of holomorphic maps to a curve $X$. For functions $u,v:\R_{\ge 0}\to \R$, the notation $u(r)\le_{exc} v(r)$ means that $u(r)\le v(r)$ holds for $r$ outside a  subset of $\R_{\ge 0}$ of finite Lebesgue measure. Similarly for $u(r)=_{exc}v(r)$. In addition, for $v$ positive valued we use Landau's notation $u(x)=o(v(x))$  to indicate that $\lim_{x\to \infty }u(x)/v(x)=0$.
\begin{theorem}[Second Main Theorem] Let $X$ be a smooth projective curve. Let $K$ be a canonical divisor on $X$ and let $A$ be an ample divisor on $X$. Let $\alpha_1,...,\alpha_q\in X$ be different points. Let $f:\C\to X$ be a holomorphic map different from the constant function $\alpha_j$ for each $j$, with image not contained in the support of $K$ and $A$. We have
$$
T_X(f,K,r) + \sum_{j=1}^q T_X(f,\alpha_j,r) \le_{exc} \sum_{j=1}^q N^{(1)}_X(f,\alpha_j,r) +O\left(\log \max\{1, T_X(f,A,r)\}\right) + o(\log r).
$$
\end{theorem}
When $f$ is constant, the result is trivial. If $f$ is non-constant, then the statement takes a simpler form, as the image of $f$ is not contained in the support of any divisor.

Due to Picard's theorem, the theorem is non-trivial only when $X$ has genus $0$ or $1$. We remark that a general second main theorem for algebraic varieties is still conjectural and pertains to the general setting of Vojta's conjectures, but we will only need the case of curves in this work.

\subsection{Meromorphic functions on $\C$} The case of $X=\Pro^1$ will be particularly relevant for us. Here we identify the Riemann sphere $\C_\infty=\C\cup\{\infty\}$ with $\Pro^1$ so that $\C$ corresponds to the affine chart $\{[1:\alpha] : \alpha\in \C\}\subseteq \Pro^1$ and $\infty$ corresponds to $[0:1]\in\Pro^1$. 

Let $\Mcal$ be the field of complex meromorphic functions on $\C$. Under the previous identifications, a function $h\in \Mcal$ can be seen as a holomorphic map $h:\C\to \Pro^1$. In this way, given $h\in \Mcal$ and a point $\alpha\in \C_\infty$ we can define $N(h,\alpha,r)$, $m(h,\alpha,r)$, $T(h,\alpha,r)$, and $N^{(1)}(h,\alpha,r)$ in the obvious way using the corresponding holomorphic map $h:\C\to \Pro^1$ (the subscript $\Pro^1$ is omitted as in this context it is clear). Furthermore,  we define 
$$
T(h,r) = T(h,\infty,r)
$$ 
and we observe that for any choice of $\alpha\in \C$, the First Main Theorem gives 
$$
T(h,r)= T(h,\alpha,r) + O(1)
$$ 
as functions of $r>0$, provided that $h$ is not the constant function $\alpha$. Also, since $-2\infty$ is a canonical divisor on $\Pro^1$, the Second Main Theorem takes the form
\begin{equation}\label{EqnSMTP1}
(q-2+o(1))T(h,r)\le_{exc} \sum_{j=1}^q N^{(1)}(h,\alpha_j,r) 
\end{equation}
where $\alpha_1,...,\alpha_q\in \C_\infty$ are different points and $f\in \Mcal$ a meromorphic function different from the constant function $\alpha_j$ for each $j$ (the error term can be made more precise if necessary).


\section{Preliminary lemmas on holomorphic maps}

\subsection{Comparison of counting functions} The next lemma will allow us to compare counting functions of various sorts.
\begin{lemma} \label{LemmaComparenN} Let $n_1(r),n_2(r):\R_{r\ge 0}\to \Z_{\ge 0}$ be functions whose points of discontinuity form a discrete set. Define 
$$
N_i(r)=\int_0^r(n_i(t)-n_i(0))t^{-1}dt + n_i(0)\log r
$$ 
for $i=1,2$. If $n_1(r)\le n_2(r)$ for all $r\ge 0$, then $N_1(r)\le N_2(r)+O(1)$.
\end{lemma}
\begin{proof} By linearity, we may assume that $n_1(r)=0$ for all $r$, and we need to show that $N_2(r)$ is bounded from below by a constant. It suffices to consider $r\ge 1$, in which case we have
$$
N_2(r)=\int_0^r(n_2(t)-n_2(0))\frac{dt}{t} +n_2(0)\log r  \ge \int_0^1(n_2(t)-n_2(0))\frac{dt}{t}.
$$
The last quantity is a constant.
\end{proof}
\subsection{Holomorphic maps to elliptic curves}

\begin{lemma} \label{LemmaSMTE} Let $E$ be a complex elliptic curve and let $\alpha\in E$. Let $f:\C\to E$ be a non-constant holomorphic map. Then
$$
T_E(f,\alpha,r) =_{exc} (1+o(1))N_E(f,\alpha,r)=_{exc} (1+o(1))N^{(1)}_E(f,\alpha,r).
$$ 
Furthermore, for every effective non-zero divisor $D$ we have
$$
T_E(f,D,r) =_{exc} (1+o(1))N_E(f,D,r).
$$
\end{lemma}
\begin{proof} A canonical divisor for $E$ is $D=0$, and the divisor $\alpha$ is ample. By the Second Main Theorem
$$
T_E(f,\alpha,r) \le_{exc} N^{(1)}_E(f,\alpha,r) + O(\log \max\{1, T_E(f,\alpha,r) \} )+ o(\log r).
$$
As $f$ is transcendental, the error term is $o(T_E(f,\alpha,r))$. The first part follows from \eqref{EqnN1NT}. The second part is deduced by additivity and the fact that positive degree divisors on curves are ample.
\end{proof}

\subsection{Meromorphic functions arising from elliptic curves}

We will be considering meromorphic functions $h\in \Mcal$ that can be written in the form $h=g\circ \phi$ with $E$ an elliptic curve over $\C$, $\phi:\C\to E$ holomorphic, and $g:E\to \Pro^1$ a rational function. Meromorphic functions $h$ of this type have better value distribution properties than general meromorphic functions.
\begin{lemma}\label{LemmaSMTfromE} Let $E$ be a complex elliptic curve, $g:E\to \Pro^1$ a non-constant morphism of degree $d$, and $\phi:\C\to E$ a non-constant holomorphic map. Let $h=g\circ\phi\in \Mcal$ and let $\alpha\in \C_\infty$. We have
\begin{equation}\label{EqnE1}
N(h,\alpha,r)=_{exc}(1+o(1))T(h,r)
\end{equation}
and
\begin{equation}\label{EqnE2}
N^{(1)}(h,\alpha,r)=_{exc}\left(\frac{\#g^{-1}(\alpha)}{d}+o(1)\right)T(h,r).
\end{equation}
\end{lemma}
\begin{proof} By functoriality of the counting function we have
$$
N(h,\alpha,r)=N(g\circ \phi,\alpha,r)=N_E(\phi,g^*\alpha,r).
$$
By Lemma \ref{LemmaSMTE} and functoriality of the height
$$
N_E(\phi,g^*\alpha,r) =_{exc} (1+o(1))T_E(\phi,g^*\alpha,r)=(1+o(1))T(h,\alpha,r).
$$
which proves \eqref{EqnE1}. 

For $z_0\in\C$ we have $\phi(z_0)\in g^{-1}(\alpha)$ if and only if $h(z_0)=\alpha$.  Together with Lemma \ref{LemmaSMTE}, this gives
$$
N^{(1)}(h,\alpha,r)=\sum_{\beta\in g^{-1}(\alpha)}N^{(1)}_E(\phi,\beta,r) =_{exc}  (1+o(1))\sum_{\beta\in g^{-1}(\alpha)}T_E(\phi,\beta,r).
$$
The divisor $g^*(\infty)$ on $E$ is ample of degree $d$, hence, it is numerically equivalent to the divisor $d\cdot \beta$ for any given point $\beta\in E$.  Lemma 3.2 in \cite{LiaoRu} allows us to compare the height for two effective, ample, numerically equivalent divisors, and we get
$$
d\cdot T_E(\phi,\beta,r)=T_E(\phi,d\cdot \beta,r) +O(1)= (1+o(1))T_E(\phi,g^*(\infty),r)
$$
from which we deduce
$$
\begin{aligned}
\sum_{\beta\in g^{-1}(\alpha)} T_E(\phi,\beta,r) &= \frac{1}{d}\sum_{\beta\in g^{-1}(\alpha)}d\cdot T_E(\phi,\beta,r) 
= \frac{1}{d}\sum_{\beta\in g^{-1}(\alpha)} (1+o(1))T_E(\phi,g^*(\infty),r)\\
& = \left(\frac{\#g^{-1}(\alpha)}{d} +o(1)\right)T_E(\phi,g^*(\infty),r) = \left(\frac{\#g^{-1}(\alpha)}{d} +o(1)\right)T(h,r).
\end{aligned}
$$
This proves \eqref{EqnE2}.
\end{proof}

\subsection{GCD counting functions} Given non-constant meromorphic functions $h_1,h_2\in \Mcal$ we define
$$
n_{GCD}(h_1,h_2,r)=\sum_{|z|\le r} \min\{\ord^+_z(h_1),\ord^+_z(h_2)\}.
$$
The \emph{GCD counting function} is
$$
N_{GCD}(h_1,h_2,r)=\int_{0}^r\left(n_{GCD}(h_1,h_2,t)-n_{GCD}(h_1,h_2,0)\right) \frac{dt}{t} + n_{GCD}(h_1,h_2,0)\log r.
$$
From Lemma \ref{LemmaComparenN} and the effectivity for proximity functions we deduce the \emph{trivial GCD bound}
$$
N_{GCD}(h_1,h_2,r)\le N(h_j,0,r) \le T(h_j,0,r) + O(1)= T(h_j,r) +O(1), \quad \mbox{ for }j=1,2.
$$
There are several works in the literature on the problem of improving this trivial bound for the GCD counting function under various assumptions, see for instance \cite{PastenWang, LiuYu, LevinWang}. For our purposes, the following will suffice.
\begin{lemma}\label{LemmaGCD} Let $\alpha_1,...,\alpha_q\in \C$ be distinct and let $h_1,h_2\in \Mcal$ be non-constant meromorphic functions. We have
$$
\sum_{j=1}^q N_{GCD}(h_1-\alpha_j,h_2-\alpha_j,r) \le T(h_1,r)+T(h_2,r) -N_{GCD}(1/h_1,1/h_2,r) + O(1).
$$
\end{lemma}
\begin{proof} For $z\in \C$ and each $j$ we have
$$
\min\{\ord^+_z(h_1-\alpha_j),\ord^+_z(h_2-\alpha_j)\}\le \ord^+_z(h_1-h_2)
$$
and
$$
\min\{\ord^+_z(1/h_1),\ord^+_z(1/h_2)\} \le \ord^+\left(\frac{1}{h_1}-\frac{1}{h_2}\right).
$$
Let $H=(h_1,h_2):\C\to \Pro^1\times \Pro^1$ and let $\Delta\subseteq \Pro^1\times \Pro^1$ be the diagonal. From the previous order estimates and the definition of the various counting functions involved, it follows that
$$
N_{GCD}(1/h_1,1/h_2,r)+\sum_{j=1}^q N_{GCD}(h_1-\alpha_j,h_2-\alpha_j,r) \le N_{\Pro^1\times\Pro^1}(H,\Delta,r).
$$
Let $\pi_1,\pi_2:\Pro^1\times \Pro^1\to \Pro^1$ be the projections onto the two factors respectively. On $\Pro^1\times\Pro^1$ we have the linear equivalence $\Delta\sim \pi_1^*\infty+\pi_2^*\infty$, and we get
$$
\begin{aligned}
N_{\Pro^1\times\Pro^1}(H,\Delta,r)& \le T_{\Pro^1\times\Pro^1}(H,\Delta,r)+O(1) &\mbox{effectivity}\\
&=T_{\Pro^1\times\Pro^1}(H,\pi_1^*\infty + \pi_2^*\infty,r)+O(1) &\mbox{First Main Theorem}\\
&=T_{\Pro^1\times\Pro^1}(H,\pi_1^*\infty,r)+T_{\Pro^1\times\Pro^1}(H,\pi_2^*\infty,r)+O(1) &\mbox{additivity}\\
&=T(h_1,r)+T(h_2,r) + O(1) &\mbox{functoriality}.
\end{aligned}
$$
\end{proof}


\section{Arithmetic progressions of holomorphic maps}

\subsection{Bound for arithmetic progressions}

In this section we prove

\begin{theorem}\label{ThmMainHolo} Let $E$ be an elliptic curve over $\C$ and let $g:E\to \Pro^1$ be a non-constant morphism of degree $d$. Let $M\ge 2$ be an integer and for $j=1,2,...,M$ let $\phi_j:\C\to E$ be non-constant holomorphic maps. Define the meromorphic functions  $f_j=g\circ \phi_j\in \Mcal$. Suppose that there are $F_1,F_2\in \Mcal$ with $F_2$ not the zero function, and pairwise distinct complex numbers $a_1,...,a_M\in\C$ such that $f_j=F_1+a_jF_2$ for each $j$. Then $M\le 10d^2-4d$.
\end{theorem}
The result will be applied in Section \ref{SecArith} in a case where the functions $f_1,...,f_j$ are distinct (not necessarily consecutive) terms of an arithmetic progression in $\Mcal$. 


\subsection{Pole computation}
\begin{lemma}\label{LemmaPoles} Let us keep the notation and assumptions of Theorem \ref{ThmMainHolo}. Let $\epsilon>0$. There are indices $i_1\ne i_2$ in $\{1,2,...,M\}$ and a Borel set $U\subseteq \R_{\ge 0}$ of infinite Lebesgue measure such that for all $r\in U$ we have $T(f_{i_2},r)\le T(f_{i_1},r)$ and
$$
 \left(1-\frac{2}{M}-\epsilon \right)\max_{1\le j\le M} T(f_j,r) \le  N_{GCD}(1/f_{i_1},1/f_{i_2}, r) \le T(f_{i_1},r) +O(1).
$$
\end{lemma}
\begin{proof} Given $z_0\in \C$, we note that if some of the $f_j=F_1+a_jF_2$ has a pole at $z_0$, then $F_1$ or $F_2$ has a pole at $z_0$. As the complex numbers $a_j$ are different, for all $1\le j\le M$ with at most one exception we get that 
$$
\ord_{z_0}(f_j)=\min\{\ord_{z_0}F_1,\ord_{z_0}F_2\} = \min_{1\le i\le M} \ord_{z_0} f_i<0.
$$
Therefore,
$$
\begin{aligned} 
\sum_{i<j} n_{GCD}(1/f_i,1/f_j,r) &=  \sum_{|z|\le r} \sum_{i<j} \min \{\ord^+_z(1/f_i),\ord^+_z(1/f_j)\}\\
&\ge \sum_{|z|\le r} \binom{M-1}{2}\max_{1\le i\le M} \ord^+_z(1/f_i). 
\end{aligned}
$$
It follows that for each $1\le i_0\le M$ we have
$$
\sum_{i<j} n_{GCD}(1/f_i,1/f_j,r) \ge  \binom{M-1}{2} n(f_{i_0},\infty, r).
$$
By Lemma \ref{LemmaComparenN} and Lemma \ref{LemmaSMTfromE}, for any given $\epsilon>0$ we get
$$
\sum_{i<j} N_{GCD}(1/f_i,1/f_j,r) \ge  \binom{M-1}{2} N(f_{i_0},\infty, r) \ge_{exc} \binom{M-1}{2}(1-\epsilon)T(f_{i_0},r).
$$
Since $i_0$ is arbitrary, we get
$$
\sum_{i<j} N_{GCD}(1/f_i,1/f_j,r) \ge_{exc} \binom{M-1}{2}(1-\epsilon) \max_{1\le i\le M} T(f_i,r).
$$
The first sum has $\binom{M}{2}$ terms. A contradiction argument shows that there are indices $i_1\ne i_2$ in $\{1,2,...,M\}$ and a Borel set $V \subseteq \R_{\ge 0}$ of infinite Lebesgue measure such that for all $r\in V$ we have
$$
N_{GCD}(1/f_{i_1},1/f_{i_2}, r)  \ge  \frac{\binom{M-1}{2}}{\binom{M}{2}}(1-\epsilon)\max_{1\le i\le M} T(f_i,r)=  \left(1-\frac{2}{M}\right)(1-\epsilon)\max_{1\le i\le M} T(f_i,r).
$$
After switching $i_1$, $i_2$ if necessary and replacing $V$ by an infinite measure subset $U$, for all $r\in U$ we also have $T(f_{i_2},r)\le T(f_{i_1},r)$. Finally, the trivial GCD bound gives
$$
N_{GCD}(1/f_{i_1},1/f_{i_2}, r) \le N(1/f_{i_1},0,r)=N(f_{i_1},\infty,r)\le T(f_{i_1},r)+O(1).
$$
We get the result adjusting $\epsilon$.
\end{proof}

\subsection{Two numerical lemmas} 

\begin{lemma}\label{LemmaMaxMin} For $x\in \R$, let us write $x^+=\max\{0,x\}$. For every $A,B\in \R$ we have
$$
(A-B)^+\ge A^+ - \min\{A^+,B^+\}.
$$
\end{lemma}
\begin{proof} This is readily checked by considering the following cases: $A\le B$; $0\ge A>B$; $A>0\ge B$; $A>B>0$. The details are left to the reader. 
\end{proof}
\begin{lemma}\label{LemmaRH} Let $E$ be a complex elliptic curve and let $g:E\to \Pro^1$ be a non-constant morphism of degree $d$. Let $\alpha_1,...,\alpha_k\in \C$ be the set of affine branch points (after identifying $\Pro^1=\C_\infty$).  We have 
$$
1\le k\le 2d\qquad
\mbox{ and }
\qquad\sum_{j=1}^k \frac{\# g^{-1}(\alpha_j)}{d} \le k-1-\frac{1}{d}.
$$
\end{lemma}
\begin{proof} The total number of branch points of a ramified morphism to $\Pro^1$ is always at least $2$, so $k\ge 1$. The Riemann-Hurwitz formula gives $(2\cdot 1-2)=d\cdot (2\cdot 0-2) + \sum_{\alpha\in \Pro^1} \left(d-\#g^{-1}(\alpha)\right)$, thus
$$
2d=\sum_{\alpha\in \Pro^1} \left(d-\#g^{-1}(\alpha)\right)\ge \sum_{\substack{\alpha\in \Pro^1\\ \alpha\, \mathrm{branch}}} 1\ge k.
$$
This proves the bounds for $k$. Finally, there is at most one branch points other than the $\alpha_j$, so the Riemann-Hurwitz formula gives
$$
2d = \sum_{\alpha\in \Pro^1} \left(d-\#g^{-1}(\alpha)\right)\le (d-1)+ \sum_{j=1}^k (d- \#g^{-1}(\alpha_j)) = (k+1)d-1 -  \sum_{j=1}^k \#g^{-1}(\alpha_j)
$$
and the result follows.
\end{proof}

\subsection{Proof of Theorem \ref{ThmMainHolo}}   Let us keep the notation and assumptions of Theorem \ref{ThmMainHolo}. Furthermore, we may assume $M\ge 4$, so that the expressions $M-1$, $M-2$, and $M-3$ are positive (this is relevant as we will eventually divide by them in some computations). 

Let $\epsilon>0$. Up to relabeling the functions $f_j$ if necessary, Lemma \ref{LemmaPoles} shows that there is a Borel set $U\subseteq \R_{\ge 0}$ of infinite Lebesgue measure such that for all $r\in U$ we have
\begin{equation}\label{Eqnf1big}
T(f_2,r)\le T(f_1,r)
\end{equation}
and
\begin{equation}\label{EqnPoles}
\left(1-\frac{2}{M}-\epsilon\right)\max_{1\le j\le M} T(f_j,r) \le  N_{GCD}(1/f_{1},1/f_{2}, r) \le T(f_{1},r) +O(1).
\end{equation}

For each $1\le j\le M$ define the complex numbers
$$
\lambda_j=\frac{a_2-a_j}{a_2-a_1}, \quad \mu_j=\frac{a_1-a_j}{a_1-a_2}\quad \gamma_j=\frac{a_j-a_2}{a_j-a_1}
$$
and observe that 
\begin{itemize}
\item All the numbers $\lambda_j,\mu_j,\gamma_j$ are non-zero.
\item The numbers $\lambda_j$ are pairwise different. Similarly for the numbers $\mu_j$ and the numbers $\gamma_j$.
\item $\lambda_j+\mu_j=1$.
\item  $\lambda_ja_1+\mu_ja_2=a_j$. 
\item $\gamma_j=-\lambda_j/\mu_j$.
\end{itemize}

Let $\alpha\in \C$. We note that
$$
\lambda_j(f_1-\alpha)+\mu_j(f_2-\alpha)= (\lambda_j+\mu_j)F_1+(\lambda_ja_1+\mu_ja_2)F_2+(\lambda_j+\mu_j)\alpha=f_j-\alpha.
$$
hence
\begin{equation}\label{Eqnlc}
\frac{f_2-\alpha}{f_1-\alpha} -\gamma_j = \mu_j^{-1}\cdot \frac{f_j-\alpha}{f_1-\alpha}.
\end{equation}
From this equation we observe that the meromorphic function $(f_2-\alpha)/(f_1-\alpha)\in \Mcal$ is not the constant function $\gamma_j$ for any $j$, since $f_j$ is not the constant function $\alpha$ ($f_j$ is non-constant).

Also from \eqref{Eqnlc} we see that given any complex number $\alpha\in \C$, for all $r\ge 0$ we have
\begin{equation}\label{EqnZeros1}
N^{(1)}\left(\frac{f_2-\alpha}{f_1-\alpha} , \gamma_j ,r\right) = N^{(1)}\left(\frac{f_j-\alpha}{f_1-\alpha} , 0 ,r \right).
\end{equation}

Given $\alpha\in \C$, let us give an upper bound for the average (for $2\le j\le M$) of the right hand side of the previous expression. Let $B[r]=\{z\in \C : |z|\le r\}$. First we observe
$$
\begin{aligned}
\sum_{j=2}^M \, & n^{(1)}\left(\frac{f_j-\alpha}{f_1-\alpha} , 0 ,r \right) \\
&\le \sum_{j=2}^M n^{(1)}\left(f_j-\alpha , 0 ,r \right)  + \sum_{j=2}^M  \#\left\{z\in B[r]: \ord_{z}(f_1-\alpha)<\ord_{z}(f_j-\alpha)\le 0\right\}.
\end{aligned}
$$
Since $f_j= F_1 + a_jF_2$, we see that if $f_1$ has a pole at some $z_0$, then all the $f_j$ have a pole of the same order at $z_0$ with at most one possible exception for $j$. Thus, given $z_0\in \C$, the condition $\ord_{z_0}(f_1-\alpha)<\ord_{z_0}(f_j-\alpha)\le 0$ holds for at most one $j$, in which case $f_1$ has a pole. We get
$$
\sum_{j=2}^M  \#\left\{z\in B[r]: \ord_{z}(f_1-\alpha)<\ord_{z}(f_j-\alpha)\le 0\right\}\le n^{(1)}(f_1,\infty,r),
$$
from which we deduce
$$
\sum_{j=2}^M  n^{(1)}\left(\frac{f_j-\alpha}{f_1-\alpha} , 0 ,r \right)\le n^{(1)}(f_1,\infty,r)+ \sum_{j=2}^M n^{(1)}\left(f_j-\alpha , 0 ,r \right).
$$
Therefore, Lemma \ref{LemmaComparenN} gives 
\begin{equation}\label{EqnZeros3}
\sum_{j=2}^M  N^{(1)}\left(\frac{f_j-\alpha}{f_1-\alpha} , 0 ,r \right)\le N^{(1)}(f_1,\infty,r)+ \sum_{j=2}^M N^{(1)}\left(f_j-\alpha , 0 ,r \right) +O(1).
\end{equation}
Let us write 
$$
T(r)=\max_{1\le j\le M} T(f_j ,r ).
$$
Using \eqref{EqnZeros3}, \eqref{EqnZeros1}, the fact that $N^{(1)}\left(f_j-\alpha , 0 ,r \right) =N^{(1)}\left(f_j,\alpha  ,r \right)$, and Lemma \ref{LemmaSMTfromE},  we deduce that for any given $\alpha\in \C$ 
\begin{equation}\label{EqnZeros2}
\begin{aligned}
\sum_{j=2}^M  N^{(1)}\left(\frac{f_2-\alpha}{f_1-\alpha} , \gamma_j ,r \right) &\le N^{(1)}(f_1,\infty,r)+ \sum_{j=2}^M N^{(1)}\left(f_j,\alpha  ,r \right) +O(1)\\
 &=_{exc}  N^{(1)}(f_1,\infty,r)+ \left(\frac{\# g^{-1}(\alpha)}{d} +o(1) \right) \sum_{j=2}^M T(f_j ,r ) \\
&\le T(f_1,r)+ \left(\frac{\# g^{-1}(\alpha)}{d} +o(1) \right) (M-1) T(r ). 
\end{aligned}
\end{equation}

 As explained after \eqref{Eqnlc}, the meromorphic function $(f_2-\alpha)/(f_1-\alpha)\in\Mcal$ is not equal to the constant function $\gamma_j$ for any $j$. The Second Main Theorem \eqref{EqnSMTP1} with the targets $\gamma_2,...,\gamma_M$ (here, $q=M-1$) gives that for any fixed $\alpha\in \C$
\begin{equation}\label{EqnSMTinProof}
(M-3+o(1))T\left(\frac{f_2-\alpha}{f_1-\alpha},r\right)\le_{exc} \sum_{j=2}^M  N^{(1)}\left(\frac{f_2-\alpha}{f_1-\alpha} , \gamma_j ,r \right).
\end{equation}
Let us give a lower bound for the expression on the left hand side of \eqref{EqnSMTinProof}. By Lemma \ref{LemmaMaxMin} we have
$$
\begin{aligned}
n\left(\frac{f_2-\alpha}{f_1-\alpha},\infty,r\right) & = \sum_{|z|\le r} \max\{0,\ord_z(f_1-\alpha)-\ord_z(f_2-\alpha)\}\\
&\ge  \sum_{|z|\le r} \ord^+_z(f_1-\alpha) - \sum_{|z|\le r}\min\{ \ord^+_z(f_1-\alpha),\ord^+_z(f_2-\alpha) \}\\
&= n(f_1,\alpha,r) - n_{GCD}(f_1-\alpha,f_2-\alpha,r).
\end{aligned}
$$
Lemma \ref{LemmaComparenN} gives the desired lower bound for the left hand side of \eqref{EqnSMTinProof}:
\begin{equation}\label{EqnDen}
  N(f_1,\alpha,r)- N_{GCD}(f_1-\alpha,f_2-\alpha,r)\le N\left(\frac{f_2-\alpha}{f_1-\alpha},\infty, r\right)+O(1)\le T\left(\frac{f_2-\alpha}{f_1-\alpha},r\right) +O(1).
\end{equation}
We conclude that for any given $\alpha\in \C$ the following holds:
$$
\begin{aligned}
(M-3+o(1)) & \left((1+o(1))T(f_1,r)- N_{GCD}  (f_1-\alpha,f_2-\alpha,r)\right) &  \\
&=_{exc}(M-3+o(1)) \left(N(f_1,\alpha,r)- N_{GCD}  (f_1-\alpha,f_2-\alpha,r)\right) & \mbox{ by Lemma \ref{LemmaSMTfromE}}\\
& \le_{exc} \sum_{j=2}^M  N^{(1)}\left(\frac{f_2-\alpha}{f_1-\alpha} , \gamma_j ,r \right)& \mbox{by \eqref{EqnDen} and \eqref{EqnSMTinProof}}\\
&\le_{exc} T(f_1,r)+ \left(\frac{\# g^{-1}(\alpha)}{d} +o(1) \right) (M-1)T( r )  & \mbox{ by \eqref{EqnZeros2}}.
\end{aligned}
$$
Rearranging and collecting the error terms, we conclude
\begin{equation}\label{EqnKeyBound}
(M-4)T(f_1,r) \le_{exc} \left(\frac{\# g^{-1}(\alpha)}{d} +o(1) \right)  (M-1)  T(r ) +(M-3)N_{GCD}  (f_1-\alpha,f_2-\alpha,r) .
\end{equation}
Let  $k$ be the number of affine branch points in $\C_\infty=\Pro^1$ of $g:\E\to\Pro^1$ and let $\alpha_1,...,\alpha_k\in \C$ be these branch points. Lemma \ref{LemmaGCD} gives
$$
\sum_{i=1}^k N_{GCD}  (f_1-\alpha_i,f_2-\alpha_i,r)\le T(f_1,r)+T(f_2,r) - N_{GCD}(1/f_1,1/f_2,r) + O(1).
$$
Using \eqref{Eqnf1big} and \eqref{EqnPoles} we get for all $r$ in the infinite measure set $ U$
$$
\sum_{i=1}^k N_{GCD}  (f_1-\alpha_i,f_2-\alpha_i,r)\le 2T(f_1,r) - \left(1-\frac{2}{M}-\epsilon\right) T(r) +O(1).
$$
Removing a finite measure subset from $U$ we get an infinite measure set $U'\subseteq U\subseteq \R_{\ge 0}$ such that for all $r\in U'$ the previous estimate holds as well as \eqref{EqnKeyBound} for $\alpha=\alpha_j$ with $1\le j\le k$. This gives that for all $r\in U'$ we have
$$
\begin{aligned}
k(M-4)T(f_1,r)&\le \left(\sum_{j=1}^k \frac{\# g^{-1}(\alpha_j)}{d} +o(1) \right)  (M-1)  T(r ) \\
&\quad +2(M-3)T(f_1,r) - (M-3)\left(1-\frac{2}{M}-\epsilon\right) T(r).
\end{aligned}
$$
Let us write 
$$
S= \sum_{j=1}^k \frac{\# g^{-1}(\alpha_j)}{d}.
$$ 
Rearranging we get
$$
\left(  \frac{M-4}{M-3}\cdot k - 2  \right)T(f_1,r) \le \left( \frac{M-1}{M-3}\cdot S +o(1) -1+\frac{2}{M}+\epsilon \right)   T(r ).
$$
Using \eqref{EqnPoles} (which is valid for $r\in U'$) we get
$$
\left(  \frac{M-4}{M-3}\cdot k - 2  \right)\left(1-\frac{2}{M}-\epsilon\right)T(r) \le \left( \frac{M-1}{M-3}\cdot S +o(1) -1+\frac{2}{M}+\epsilon \right)   T(r ).
$$
Since $U'\subseteq \R_{\ge 0}$ has infinite measure, we can let $r\to+\infty$ over a sequence in $U'$. As the functions $f_j$ are non-constant, we get $T(r)\to +\infty$ over this sequence, and we deduce
$$
\left(  \frac{M-4}{M-3}\cdot k - 2  \right)\left(1-\frac{2}{M}-\epsilon\right) \le \frac{M-1}{M-3}\cdot S -1+\frac{2}{M}+\epsilon .
$$
Since $\epsilon>0$ is arbitrary and $S\le k-1-1/d$ (cf. Lemma \ref{LemmaRH}) we obtain
$$
\left(  \frac{M-4}{M-3}\cdot k - 2  \right)\left(1-\frac{2}{M}\right) \le \frac{M-1}{M-3}\left(k-1-\frac{1}{d}\right)  -1+\frac{2}{M}.
$$

If $M$ were very large, this would approximately give $k-2\le k-1-1/d-1=k-2-1/d$ which is not possible. So, it is clear that this expression constrains the size of $M$. Let us work out the precise details in order to get the desired bound: Rearranging we obtain
\begin{equation}\label{EqnFinal}
\frac{(M-1)M}{d}  + 2(2M-3) - (5M-8) \cdot k \le 0.
\end{equation}
The quadratic function 
$$
u(t)=t(t-1)/d + 2(2t-3) - (5t-8)k
$$ 
is increasing for $t\ge t_0= (1+(5k-4)d)/2$ and satisfies $u((5k-4)d)=3k-2\ge 1$. Since $(5k-4)d\ge t_0$ we deduce that $u(t)\ge 1$ for $t\ge (5k-4)d$. Therefore, \eqref{EqnFinal} with Lemma \ref{LemmaRH} shows that 
$$
M\le (5k-4)d\le 10d^2-4d.
$$
This concludes the proof of Theorem \ref{ThmMainHolo}. \qed
\section{Some geometric constructions}\label{SecGeom}

\subsection{Notation and first constructions} Let $k$ be an algebraically closed field of characteristic $0$, let $E$ be an elliptic curve over $k$ and let $n$ be a positive integer.  Let $g\in k(E)$ be a non-constant rational function of degree $d$. We identify $\A^1_k$ with the affine chart $\{[x_0:x_1]\in \Pro^1_k : x_0\ne 0\}$ of $\Pro^1_k$. In particular, $g$ can be identified with a morphism $g:E\to\Pro^1_k$ of degree $d$ defined over $k$. We consider the abelian variety $A=E^n$ of dimension $n$. Let $G_n:A\to (\Pro^1)^n$ be the morphism obtained from $n$ copies of $g$. 
\begin{lemma}\label{LemmaFlat}
 The morphism $G_n$ is finite of degree $d^n$ and flat.
 \end{lemma}
\begin{proof}
The map $g:E\to \Pro^1_k$ is surjective and finite of degree $d$. Hence,  $G_n:E^n\to (\Pro^1_k)^n$ is surjective and  finite of degree $d^n$.  On the other hand, the map $g$ is flat by \cite{Hartshorne} III Prop.\ 9.7. Hence, repeated applications of \cite{Hartshorne} III Prop.\ 9.2 give that $G$ is flat. \end{proof}
\subsection{The surfaces $U_n$ and $H_n$} Let us assume that $n\ge 3$. Let $u_1,...,u_n$ be the coordinates on $\A^n_k$ and let us define the affine variety
\begin{equation}\label{EqnUn}
U_n:\left\{\begin{aligned}
u_3-2u_2+u_1&=0\cr
&\vdots\cr
u_n-2u_{n-1}+u_{n-2}&=0
\end{aligned}\right.\subseteq \mathbb{A}^n_k.
\end{equation}
Under the previously chosen inclusion $\A^1_k\subseteq \Pro^1_k$, we have $\A^n_k\subseteq (\Pro^1_k)^n$. Let $H_n$ be the Zariski closure of $U_n$ in $(\Pro^1_k)^n$. Let $p_j:(\Pro^1_k)^n\to \Pro^1_k$ be the $j$-th coordinate projection.
\begin{lemma}\label{LemmaSurfacesDown}
We have that $U_n$ is a linear surface in $\A_k^n$ and $H_n$ is an irreducible  projective surface. Furthermore, for every $j$ we have that $p_j$ restricts to a surjective map $H_n\to\Pro^1_k$.
\end{lemma}
\begin{proof} $U_n$ is a linear surface because the $n-2$ linear equations defining it are linearly independent. The other claims follow.
\end{proof}

\begin{lemma} \label{LemmaIsomHn}  Let $1\le i<j\le n$.  The projection $p_{ij}: (\Pro^1_k)^n\to (\Pro^1_k)^2$ onto the coordinates $i$ and $j$ restricts to a map $H_n\to (\Pro^1_k)^2$ which is finite of degree $1$ above $\A^2_k$.
\end{lemma}
\begin{proof} From the equations of $U_n$, we note that if a point $(\alpha_1,...,\alpha_n)\in H_n$ has some coordinate $\alpha_j=\infty\in \Pro^1_k$, then all other coordinates with at most one exception are also equal to $\infty$. Thus, the preimage of $\A^2_k$ under $p_{ij}|_{H_n}:H_n\to (\Pro^1_k)^2$ is precisely $U_n$.  Finally, since $k$ has characteristic $0$, fixing $u_i$ and $u_j$ in $k$ (with $i\ne j$) determines a unique point in $U_n$, namely
$$
u_\ell =\frac{\ell-j}{i-j}\cdot  u_i + \frac{\ell-i}{j-i}\cdot u_j, \quad 1\le \ell \le n.
$$
\end{proof}

\subsection{The surfaces $V_n$ and $X_n$} Let us define 
$$
X_n=G^{-1}_n(H_n)\subseteq A \quad \mbox{ and }\quad V_n = G^{-1}_n(U_n)\subseteq X_n.
$$
\begin{lemma}\label{LemmaSurfacesUp}
We have that $X_n$ is a projective surface and $V_n$ is a dense open subset in $X_n$.  Moreover, the morphism $G_n:E^n\to (\Pro^1_k)^n$ restricts to a morphism $G'_n:X_n\to H_n$ which is surjective, finite of degree $d^n$, and flat. 
 \end{lemma}
\begin{proof}
Since $G_n$ is flat (cf.\ Lemma \ref{LemmaFlat}), it is open by \cite{Hartshorne} III Exer.\ 9.1. It follows that $G^{-1}_n(cl(S))=cl(G^{-1}_n(S))$ for every set $S\subseteq (\Pro^1_k)^n$, where $cl(-)$ denotes Zariski closure. As $cl(U_n)=H_n$, we get that $X_n$ is the Zariski closure of $V_n$.

 The branch divisor of $G_n$ is $\sum_{j=1}^n p_j^* B_g$ where $B_g\subseteq \Pro^1_k$ is the branch divisor of $g$. From Lemma \ref{LemmaSurfacesDown} we deduce that $H_n$ is not contained in the branch locus of $G_n$. It follows that $G_n$ restricts to a finite surjective map $G'_n:X_n\to H_n$ of degree $d^n$. We note that $V_n=G^{-1}_n(U_n)=(G')^{-1}_n(U_n)$ and $U_n$ is open in $H_n$, thus $V_n$ is open.
 
 Finally, note that $G'_n$ is the base change of $G_n$ by the closed immersion $H_n\to (\Pro^1_k)^n$, hence $G'_n$ is flat (cf.\ \cite{Hartshorne} III Prop.\ 9.2). As $G_n$ has pure relative dimension $0$, we obtain from Lemma \ref{LemmaSurfacesDown} and \cite{Hartshorne} III Coro.\ 9.6 that $\dim X_n=\dim H_n=2$. 
\end{proof}

\subsection{A line sheaf on $E^n$} Let $\pi_j:E^n\to E$ be the $j$-th coordinate projection. Let $e_E$ be the neutral point of $E$ and consider the following line sheaf on $E^n$:
$$
\Lcal_n=\Ocal\left(\sum_{j=1}^n\pi_j^*e_E\right).
$$
\begin{lemma} \label{LemmaAmpleSym} The line sheaf $\Lcal_n$ on the abelian variety  $E^n$ is ample and symmetric. 
\end{lemma}
\begin{proof} For $m\in \Z$ and $B$ an abelian variety, we write $[m]_B$ for the endomorphism of multiplication by $m$ on $B$. Let $A=E^n$ as before. We have $[-1]_E^*e_E=e_E$, hence
$$
\begin{aligned}[]
[-1]_A^*\sum_{j=1}^n\pi_j^*e_E &= \sum_{j=1}^n[-1]_A^*\pi_j^*e_E=\sum_{j=1}^n(\pi_j[-1]_A)^*e_E\\
&= \sum_{j=1}^n([-1]_E\pi_j)^*e_E= \sum_{j=1}^n\pi_j^*[-1]_E^*e_E=\sum_{j=1}^n\pi_j^*e_E.
\end{aligned}
$$
It follows that $\Lcal_n$ is symmetric on $E^n$. Since $\Lcal_n\simeq \bigotimes_{j=1}^n \pi_j^*\Ocal(e_E)$ and $\Ocal(e_E)$ is ample on $E$ (it has degree $1$), we get that $\Lcal_n$ is ample on $E^n$. 
\end{proof}

\subsection{Degree estimates} Given a line sheaf $\Fcal$ on a smooth projective variety $Y$ and a closed set $Z\subseteq Y$, we define $\deg_\Fcal Z$ as $\deg([\Fcal]^{\dim Z}\cdot [Z])$ if $Z$ is irreducible, and we extend this definition linearly for general $Z$. Here, the intersection product occurs in the Chow ring $Ch(Y)=\oplus_j Ch^j(Y)$ of $Y$ (graded by codimension) and $\deg: Ch_0(Y)\to\Z$ is the usual degree map on $0$-cycles ---we use the standard convention that $Ch^j$ denotes codimension $j$ cycles, while $Ch_j$ denotes cycles of dimension $j$. For instance, see Appendix A in \cite{Hartshorne} for a survey of intersection theory.
\begin{lemma}\label{LemmaDegXn} We have $\deg_{\Lcal_n} X_n \leq (n^2-n)d^{2n-2}$.
\end{lemma}
\begin{proof} For the following computations, let us recall that if $f:Y\to Z$ is a morphism of smooth projective varieties over $k$, then the pull-back $f^*:Ch(Z)\to Ch(Y)$ is a graded ring morphism, while the push-forward $f_*:Ch(Y)\to Ch(Z)$ respects addition and shifts the grading.

As $X_n$ is a surface, we have $\deg_{\Lcal_n} X_n=\deg([\Lcal_n]^{2}\cdot[X_n])$. We expand the intersection product
$$
\begin{aligned}[]
[\Lcal_n]^{2}\cdot[X_n] & =\sum_{i=1}^n\sum_{j=1}^n [\pi_i^*e_E]\cdot [\pi_j^*e_E]\cdot [X_n]\\
&=\sum_{i=1}^n [\pi_i^*e_E]^2\cdot [X_n] + 2\sum_{1\le i<j\le n} [\pi_i^*e_E]\cdot [\pi_j^*e_E]\cdot [X_n]
\end{aligned}
$$
Moving $e_E$ on $E$, we see that  $[\pi_i^*e_E]^2=0\in Ch^2(A)$. On the other hand, $[X_n]=G^*_n[H_n]\in Ch_2(A)$ by Lemma \ref{LemmaSurfacesUp}. Hence, the projection formula gives the following identities in $Ch_0((\Pro^1_k)^n)$
$$
\begin{aligned}
(G_n)_*\left([\Lcal_n]^{2}\cdot[X_n] \right)&= 2\sum_{1\le i<j\le n} (G_n)_*\left( [\pi_i^*e_E]\cdot [\pi_j^*e_E]\cdot G^*_n[H_n]\right)\\
&= 2\sum_{1\le i<j\le n} (G_n)_*([\pi_i^*e_E]\cdot [\pi_j^*e_E])\cdot [H_n]\\
&= 2\sum_{1\le i<j\le n} (G_n)_*([\pi_{ij}^*((e_E,e_E))])\cdot [H_n]
\end{aligned}
$$
where $\pi_{ij}:E^n\to E^2$ is the projection onto the $i$ and $j$ coordinates. 

Note that $\pi_{ij}^*((e_E,e_E))$ is obtained from $E^n$ by replacing the copies of $E$ in the coordinates $i$ and $j$ by $\{e_E\}$. Let $p_{ij}:(\Pro^1_k)^n\to (\Pro^1_k)^2$ be the projection onto the coordinates $i$ and $j$. We deduce that the map $\pi_{ij}^*((e_E,e_E))\to p_{ij}^*((g(e_E),g(e_E)))$ induced by $G_n$ is (up to the obvious isomorphisms) the same as $G_{n-2}:E^{n-2}\to (\Pro^1_k)^{n-2}$, which has degree $d^{n-2}$ by Lemma \ref{LemmaFlat}.  This gives
$$
(G_n)_*([\pi_{ij}^*((e_E,e_E))]) = d^{n-2} [p_{ij}^*((g(e_E),g(e_E))]\in Ch^2((\Pro^1_k)^n).
$$
Choose a $k$-rational point $x\in \A^1_k\subseteq \Pro^1_k$. For $i<j$, Lemma \ref{LemmaIsomHn} gives the following on $(\Pro^1_k)^n$
$$
\deg\left( [p_{ij}^*((g(e_E),g(e_E))]\cdot [H_n]\right)=\deg\left([p_{ij}^*((x,x))]\cdot [H_n]\right)=1.
$$
From here we deduce
$$
\deg \left((G_n)_*\left([\Lcal_n]^{2}\cdot[X_n] \right) \right)= (n^2-n)d^{n-2}.
$$
As $G_n:A\to (\Pro^1_k)^n$ is finite of degree $d^n$ (cf.\ Lemma \ref{LemmaFlat}), we get the desired bound.
\end{proof}


\section{Arithmetic progressions for finite rank groups} \label{SecArith}

\subsection{Main arithmetic result} Let $L$ be a field. An \emph{arithmetic progression} in $L$ is a sequence $u_1,...,u_n$ of elements of $L$ such that for some $a,b\in L$ we have $u_j=a+jb$ for each $j=1,...,n$. We say that the arithmetic progression $u_1,...,u_n$ is \emph{trivial} if al the terms $u_j$ are equal, i.e.\ $b=0$. Otherwise, the arithmetic progression is said to be \emph{non-trivial}.

The rank of an abelian group $\Gamma$ is defined as 
$$
\rk \Gamma = \dim_\Q(\Gamma\otimes_\Z\Q).
$$
In particular, if $\Gamma$ is a torsion abelian group, then $\rk \Gamma=0$.

\begin{theorem}\label{ThmMainArith} Let $j_0\in \Q^{alg}$ and let $d\ge 2$ be an integer. There is an effectively computable constant $\kappa(j_0,d)$ depending only on $j_0$ and $d$ such that the following holds:

 Let $E$ be an elliptic curve over $\Q^{alg}$ with $j$-invariant equal to $j_0$. Let $g\in k(E)$ be a non-constant rational function on $E$ of degree $d$ defined over $\Q^{alg}$. Let $\Gamma\subseteq E(\Q^{alg})$ be a subgroup of finite rank. Suppose that for a positive  integer $N$  there is  a sequence $P_1,...,P_N$ of points in $\Gamma$ such that no $P_j$ is a pole of $g$, and the sequence $g(P_1),...,g(P_N)\in \Q^{alg}$ is a non-trivial arithmetic progression. Then
$$
1+ \rk \Gamma > \kappa(j_0,d)\cdot \log N.
$$
\end{theorem}

We remark that if $L$ is a field of positive characteristic $p>0$, then a non-trivial arithmetic progression in $L$ can have repeated terms. However, if $L$ has characteristic $0$, then an arithmetic progression is non-trivial (i.e.\ not all terms are the same) if and only if all its terms are different. For our purposes, we will work in characteristic zero. 

We will need the following characterization of arithmetic progressions.

\begin{lemma}\label{LemmaAP2ndDiff} Let $n\ge 3$ be an integer, let $L$ be a field, and let $u_1,...,u_n\in L$. The sequence $u_1,...,u_n$ is an arithmetic progression if and only if $u_{j}-2u_{j-1} +u_{j-2}=0$ for each $3\le j\le n$. 
\end{lemma}
\begin{proof} Arithmetic progressions satisfy the required equations since one directly checks
$$
(a+jb)-2(a+(j-1)b)+(a+(j-2)b)=0.
$$
Conversely, if the sequence $u_1,...,u_n$ satisfies $u_{j}-2u_{j-1} +u_{j-2}=0$ for each $3\le j\le n$, then inductively one proves that $u_j=a+jb$ with $a=2u_1-u_2$ and $b=u_2-u_1$.
\end{proof}

\subsection{Lang's conjecture after Vojta, Faltings, and R\'emond}  In  \cite{Faltings2,Faltings3}  Faltings proved  Lang's conjecture on rational points in sub-varieties of abelian varieties. Namely, if $L$ is a number field, $A$ is an abelian variety over $L$ and $X\subseteq A$ is a sub-variety defined over $L$, then all but finitely many $L$-rational points of $X$ are contained in the Kawamata locus of $X$; i.e., the union of translates of positive dimensional abelian sub-varieties of $A$ contained in $X$. (We recall that the Kawamata locus is Zariski closed by a theorem of Kawamata \cite{Kawamata}.)
This proof extended ideas of Vojta's proof \cite{VojtaMordell} of Falting's theorem for curves  \cite{Faltings1}. See also Bombieri's simplification \cite{Bombieri} of Vojta's argument.
 
Faltings theorem on sub-varieties of abelian varieties has been extended in several directions. We need a quantitative generalization due to R\'emond \cite{Remond1, Remond2}, which also extends Raynaud's theorem on torsion points \cite{Raynaud} (i.e.\ the Manin-Mumford conjecture).

\begin{theorem}[R\'emond]\label{ThmRemond}
Let $A$ be an abelian variety of dimension $n$ defined over $\Q^{alg}$, and let $\Lcal$ be a symmetric ample invertible sheaf on $A$. There is an effectively computable number $c(A,\Lcal)>0$ such that the following holds: 

Let $X$ be a closed subvariety of $A$ of dimension $m$, and let $\Lambda$ be a subgroup of $A(\Q^{alg})$ such that its rank $r=\dim_\mathbb{Q}(\Lambda\otimes_\mathbb{Z} \mathbb{Q})$ is finite. There is a non-negative integer 
$$
R\leq \left(c(A,\Lcal)\deg_\Lcal X \right)^{(r+1)n^{5(m+1)^2}}
$$
and there exist points $x_1,\ldots,x_R$ in $X(\Q^{alg})\cap\Lambda$ and abelian subvarieties $T_1,\ldots,T_R$ of $A$ satisfying that $x_i+T_i\subseteq X$ for each $1\leq i\leq R$, and 
$$
X(\Q^{alg})\cap \Lambda=\bigcup_{i=1}^R(x_i+T_i)(\Q^{alg})\cap\Lambda.
$$
\end{theorem}

This formulation is the same as Th\'eor\`eme 1.2 in \cite{Remond1}, with the additional remark that the number $c(A,\Lcal)$ can be effectively computed. In fact, this point is explained in \emph{loc.\ cit.} after the statement of Th\'eor\`eme 1.2, and the precise details are given in Th\'eor\`eme 2.1 and the paragraph after it.

Moreover, a simple closed formula for $c(A, \Lcal)$  is given in Th\'eor\`eme 1.3 of \cite{Remond2} using the notion of \emph{theta height} of $A$, under the assumption that $\Lcal$ induces a principal polarization.  See Section \ref{SecPazuki} for details on how to use these explicit effective estimates in our context. 

\subsection{Proof of Theorem \ref{ThmMainArith}} \label{SecPfMainArith} Let us keep the notation and assumptions of Theorem \ref{ThmMainArith}. Let us consider constructions from Section \ref{SecGeom} with $k=\Q^{alg}$, $n=10d^2-4d+2$, and the choice of $E$ and $g$ given in Theorem \ref{ThmMainArith}. Especially, we obtain the morphism $G_n:E^n\to (\Pro^1_k)^n$, the projective surfaces $H_n\subseteq (\Pro^1_k)^n$ and $X_n\subseteq E^n$, the open sets $U_n\subseteq H_n$ and $V_n\subseteq X_n$, and the line sheaf $\Lcal_n$ on $E^n$.

Let $\Delta_n=\{u_1=u_2=...=u_n\}\subseteq \A^n_k$ be the diagonal line. We observe that $\Delta_n$ is a Zariski closed set in $U_n$. Let us define
$$
U_n^0=U_n- \Delta_n\quad \mbox{ and }\quad V_n^0=G^{-1}_n(U_n^0)\subseteq V_n.
$$
\begin{lemma} \label{LemmaNonTrivLocus} Let $L/k$ be a field extension. Let $\alpha_1,...,\alpha_n\in L$. We have that the sequence is an arithmetic progression in $L$ if and only if $(\alpha_1,...,\alpha_n)\in U_n(L)$. In this case, the arithmetic progression is non-trivial if and only if $(\alpha_1,...,\alpha_n)\in U_n^0(L)$.

Furthermore, let $P_1,...,P_n$ be a sequence of points in $E(L)$. We have that $g(P_1),...,g(P_n)$ is an arithmetic progression in $L=\A^1_k(L)$ if and only if $(P_1,...,P_n)\in V_n(L)$. In this case, the arithmetic progression is non-trivial if and only if  $(P_1,...,P_n)\in V_n^0(L)$.
\end{lemma}
\begin{proof} The sequence $\alpha_1,...,\alpha_n$ is an arithmetic progression if and only if it has second differences equal to $0$ (cf. Lemma \ref{LemmaAP2ndDiff}). This is equivalent to the condition that $(\alpha_1,...,\alpha_n)\in U_n(L)$.  The sequence is trivial if and only if all terms are equal, which is equivalent to $(\alpha_1,...,\alpha_n)\in \Delta_n(L)$.

The second part follows from the first part, using  $V_n=G^{-1}_n(U_n)$ and $V^0_n=G^{-1}_n(U^0_n)$.
\end{proof}
Note that $V_n^0$ is a non-empty open set of $V_n$, thus, of $X_n$. Let $Z_n=X_n-V_n^0$; this is a proper Zariski closed subset of $X_n$. We now show that the Kawamata locus of $X_n$ is contained in $Z_n$.

\begin{lemma}\label{LemmaKawamataLocus} Let $T\subseteq E^n$ be an  abelian sub-variety of strictly positive dimension, and suppose that $x\in X_n(k)$ satisfies $x+T\subseteq X_n$. Then $x+T\subseteq Z_n$.
\end{lemma}
\begin{proof} Let $T'=x+T$. Since $T'(\C)$ is a positive dimensional complex torus, there is a non-constant holomorphic map $\phi:\C \to X_n$ whose image is Zariski dense in $T'$; this can be seen by considering $T'\simeq \C^g/\Lambda$ for a lattice $\Lambda\subseteq \C^g$.  Let us write $\phi_j=\pi_j\circ \phi:\C\to E$, so that $\phi=(\phi_1,...,\phi_n)$.

By contradiction, suppose that $T'$ is not contained in $Z_n$. Then the image of $\phi$ meets $V_n^0$. Thus, the image of $G_n\circ \phi$ meets $U_n^0$. In particular, all the compositions $f_j=g\circ \phi_j$ are complex meromorphic functions (i.e.\ $\phi_j$ is not identically a pole of $g$, for each $j$).

Since the image of $G_n\circ \phi=(f_1,...,f_n)$ meets the Zariski open set $U_n^0\subseteq (\Pro^1_k)^n$, we see that for all but countably many $z_0\in\C$ we have $(f_1(z_0),...,f_n(z_0))\in U_n^0(\C)$. Thus, by the identity principle, we get that $(f_1,...,f_n)\in\Mcal^n$ satisfies the equations defining $U_n$ but not the equations defining $\Delta_n$. This means that $(f_1,...,f_n)\in U_n^0(\Mcal)$ as an $\Mcal$-rational point. By Lemma \ref{LemmaNonTrivLocus} with $L=\Mcal$, we get that $f_1,...,f_n$ is a non-trivial arithmetic progression in $\Mcal$. Hence, there are $F_1,F_2\in \Mcal$ such that $F_2$ is not the zero function and $f_j=F_1+jF_2$ for each $1\le j\le n$.

As $\phi$ is non-constant and $g$ is finite, at least one of the $f_j$ is non-constant.  Thus, at least one of $F_1$ or $F_2$ is non-constant, and it follows that at most one of the $f_1,...,f_n$ can be constant. Relabeling if necessary and deleting one term, we apply Theorem \ref{ThmMainHolo} with $M=n-1$ to conclude that $M\le 10d^2-4d$. Since  $n=10d^2-4d+2$ we get $10d^2-4d+1\le 10d^2-4d$, a contradiction.
\end{proof}
Finally, we proceed to conclude the proof of Theorem \ref{ThmMainArith}.

We apply Theorem \ref{ThmRemond} with $A=E^n$ and $\Lcal=\Lcal_n$; this choice of sheaf is allowed by Lemma \ref{LemmaAmpleSym}. We obtain the effectively computable constant $c(E^n,\Lcal_n)>0$ provided by Theorem \ref{ThmRemond}. This constant only depends on the isomorphism class of $E$ over $k=\Q^{alg}$ and our choice $n=10d^2-4d+2$. So, $c(E^n,\Lcal_n)$ only depends on $d$ and $j_0$, the $j$-invariant of $E$, and this dependence is effective.  Let us write $c(j_0,d)$ instead of $c(E^n,\Lcal_n)$ to make explicit that this quantity only depends on $j_0$ and $d$.

We take $X=X_n$, which has dimension $m=2$ (cf.\ Lemma \ref{LemmaSurfacesUp}). Let us consider the group $\Lambda = \Gamma\times\cdots\times  \Gamma\subseteq E^n(k)$ with $\Gamma$ as in Theorem \ref{ThmMainArith} and observe that $\Lambda$ has finite rank 
$$
r=\rk \Lambda=n\cdot \rk\Gamma.
$$
By Lemma \ref{LemmaDegXn}, the number $R$ provided by Theorem \ref{ThmRemond} satisfies
\begin{equation}\label{EqnRbound}
R \le (c(j_0,d)\deg_{\Lcal_n} X_n)^{n^{45}(r+1)}\leq R_0:= \left(c(j_0,d) \cdot (n^2-n)d^{2n-2}\right)^{n^{45}(1 + n\cdot\rk \Gamma)}.
\end{equation}
In addition, there are points $x_1,...,x_R\in X_n(k)$ and abelian sub-varieties $T_1,...,T_R\subseteq E^n$ such that 
$$
V_n^0(k)\cap \Lambda \subseteq X_n(k)\cap \Lambda \subseteq \bigcup_{i=1}^R (x_i+T_i)(k).
$$
By Lemma \ref{LemmaKawamataLocus}, all the $T_i$ with $\dim T_i\ge 1$ satisfy $x_i+T_i\subseteq Z_n$. Thus, writing 
$$
I=\{1\le i\le R : T_i=\{e_{A}\}\}
$$ 
we get
$$
V_n^0(k)\cap \Lambda \subseteq \bigcup_{i\in I} (x_i+T_i)(k) = \{x_i : i\in I\}.
$$
In particular, 
\begin{equation}\label{EqnPtsbound}
\# \left(V_n^0(k)\cap \Lambda\right) \le R.
\end{equation}
Let us define 
\begin{equation}\label{EqnCj0d}
C(j_0,d)=\left(c(j_0,d) \cdot (n^2-n)d^{2n-2}\right)^{n^{46}}\quad \mbox{with}\quad n=10d^2-4d+2.
\end{equation}
By \eqref{EqnRbound} and \eqref{EqnPtsbound}, we see that
\begin{equation}\label{EqnKeyBoundArith}
C(j_0,d)^{1+\rk \Gamma}=R_0\left(c(j_0,d) \cdot (n^2-n)d^{2n-2}\right)^{n^{46}-n^{45}} > 2R_0 > n+R_0\geq n + \# \left(V_n^0(k)\cap \Lambda\right).
\end{equation}
Let $P_1,...,P_N\in \Gamma$ be as in the statement of Theorem \ref{ThmMainArith}, i.e, $g(P_1),...,g(P_N)$ is a non-trivial arithmetic progression in $k$. We note that for each $j=1,...,N-n$ the sequence $g(P_{j}),g(P_{j+1}),..., g(P_{j+n-1})$ is a non-trivial arithmetic progression in $k$ of length $n$, and all these $N-n$ sequences are different. From Lemma \ref{LemmaNonTrivLocus} we deduce that $(P_j,...,P_{j+n-1})\in V_n^0(k)$ for each $j=1,...,N-n$ and these are different points as their images under $G_n$ are different. Furthermore, by assumption $P_i\in \Gamma$ for each $i$, so we get $(P_j,...,P_{j+n-1})\in V_n^0(k)\cap \Lambda$. This proves
$
\# \left(V_n^0(k)\cap \Lambda\right)\ge N-n.
$

Together with \eqref{EqnKeyBoundArith} we finally get 
$$
C(j_0,d)^{1+\rk \Gamma}> N.
$$
This proves Theorem \ref{ThmMainArith} with 
\begin{equation}\label{Eqnkj0d}
\kappa(j_0,d) = 1/\log C(j_0,d). 
\end{equation}
Since $c(j_0,d)$ is effectively computable, so are $C(j_0,d)$ and $\kappa(j_0,d)$. \qed

\subsection{Consequences}\label{SecConsequences}

Let us formulate here two direct consequences of Theorem \ref{ThmMainArith} that relate arithmetic progressions of rational points to two different aspects of the arithmetic of elliptic curves: the \emph{rank}, and on the other hand, the \emph{torsion} part.

\begin{corollary}\label{Coro1} Let $j_0\in \Q^{alg}$ and let $d$ be a positive integer. There is an effectively computable constant $\kappa(j_0,d)>0$ depending only on $j_0$ and $d$ such that the following holds:

Let $L\subseteq \Q^{alg}$ be a number field containing $j_0$ and let $E$ be an elliptic curve over $L$ with $j$-invariant equal to $j_0$. Let $g$ be a non-constant rational function on $E$ defined over $L$ of degree $d$. If for some $N$ there is a sequence of points $P_1,...,P_N\in E(L)$ such that $g(P_1),...,g(P_N)$ is a non-trivial arithmetic progression in $L$, then
$$
1+\rk E(L)> \kappa(j_0,d)\cdot \log N
$$
\end{corollary}
\begin{proof} All elliptic curves over number fields with $j$-invariant equal to $j_0$ are isomorphic to each other after base change to $\Q^{alg}$.  Thus, the result is immediate from Theorem \ref{ThmMainArith} applied to $E'=E\otimes_L \Q^{alg}$ choosing  the group $\Gamma=E(L)$, which is a group of finite rank by the Mordell-Weil theorem. Here, we use the inclusion $\Gamma \subseteq E(\Q^{alg})\simeq E'(\Q^{alg})$. 
\end{proof}

\begin{corollary} \label{Coro2} Let $E$ be an elliptic curve over $\Q^{alg}$ and let $d$ be a positive integer. There is an effectively computable constant $N(E,d)$ depending only on $E$ and $d$ such that the following holds:

Let $g$ be a non-constant rational function on $E$ defined over $\Q^{alg}$ of degree $d$. Let $S_g\subseteq E(\Q^{alg})$ be the set of poles of $g$ and let $E(\Q^{alg})_{tor}$ be the group of all torsion points of $E$. The set 
$$
g\left(E(\Q^{alg})_{tor}-S_g\right)\subseteq \Q^{alg}
$$ 
does not contain non-trivial arithmetic progressions of length greater than $N(E,d)$.
\end{corollary}
\begin{proof} The group $\Gamma=E(\Q^{alg})_{tor}$ has rank $0$ and the isomorphism class of $E$ over $\Q^{alg}$ is determined by the $j$-invariant. Thus, the result is a direct consequence of Theorem \ref{ThmMainArith}.
\end{proof}

As a special case, we obtain two of the main results stated in the Introduction.
\begin{proof}[Proof of Theorem \ref{ThmIntro}] The result follows from Corollary \ref{Coro1} with $d=2$ for $x$-coordinates, $d=3$ for $y$-coordinates, and choosing $L=\Q$.
\end{proof}
\begin{proof}[Proof of Theorem \ref{ThmTorIntro}] The result follows from Corollary \ref{Coro2} with $d=2$ for $x$-coordinates and $d=3$ for $y$-coordinates.
\end{proof}
\subsection{Effectivity}\label{SecPazuki} Here we prove that \eqref{Eqncj0} gives an admissible value for $c(j_0)$ in Theorem \ref{ThmIntro}. Although we restrict ourselves to the setting of Theorem \ref{ThmIntro} for the sake of simplicity, it will be clear from the argument that a similar (although lengthier) computation gives an explicit value for the effective constants in Theorem \ref{ThmMainArith} and its consequences.

Let $j_0\in \Q$. From the proof of Theorem \ref{ThmIntro} (cf. Section \ref{SecConsequences}) we note that  $c(j_0)$ can be chosen as any value smaller than 
$$
 \min\{\kappa(j_0,2), \kappa(j_0,3)\} = \frac{1}{\log \max\{C(j_0,2), C(j_0,3)\}} = \frac{1}{\log C(j_0,3)} 
$$
with $\kappa(j_0,d)$ and $C(j_0,d)$ as defined in  \eqref{EqnCj0d} and \eqref{Eqnkj0d}. Here we used that for $j_0$ fixed, one can check that the quantity $C(j_0,d)$ is increasing on $d$. We compute
\begin{equation}\label{EqnExpl0}
\log C(j_0,3)=80^{46}\left(\log(c(j_0,3))+\log(6320)+158\log(3)\right) <80^{46}\left(\log(c(j_0,3))+183 \right)
\end{equation}
where $c(j_0,3)$ is as in Section \ref{SecPfMainArith}. Namely, $c(j_0,3)=c(E^n,\Lcal_n)$ where $n=10\cdot 3^2-4\cdot 3+2=80$, $E$ is an elliptic curve over $\Q^{alg}$ with $j$-invariant equal to $j_0$,  and $c(A,\Lcal)$ is the constant appearing in Theorem \ref{ThmRemond}.

An admissible value for $c(A,\Lcal)$ is given in  Th\'eor\`eme 1.3 of \cite{Remond2}  under the additional assumption that $\Lcal$ induces a principal polarization. In our case, $\Lcal=\Lcal_n\simeq \bigotimes_{j=1}^n \pi_j^*\Ocal(e_E)$ induces a principal polarization on $E^n$, see \cite{Lange}. Therefore, the formula from \cite{Remond2} directly applies. In our case, the abelian variety $A=E^n$ and the sheaf $\Lcal_n$ can be defined over $\mathbb{Q}$ because $j_0\in\Q$. Therefore, Th\'eor\`eme 1.3 of \cite{Remond2} allows us to take 
\begin{equation}\label{EqnExpl1}
c(j_0,3)=c(E^n,\Lcal_n)=2^{34}\cdot\max\{1,h_\Theta(E^n)\},\quad n=80
\end{equation}  
where $h_\Theta$ denotes the \emph{Theta height} associated to the line sheaf $\Lcal^{\otimes 16}$. Let us estimate the Theta height. First we compare it to the semi-stable Faltings' height $h_F$ using results by Pazuki, namely Corollary 1.3(2) in \cite{Pazuki}. Here we use the normalization of $h_F$ used in \emph{loc. cit.} As we are using the Theta height associated to $\Lcal^{\otimes 16}$, we must choose $r=4$ in \cite{Pazuki} Corollary 1.3(2), which gives
$$
\begin{aligned}
h_\Theta(E^n)&\leq \frac{1}{2}\max\{1,h_F(E^n)\}+C_2(80,4)\log\left(2+\max\{1,h_F(E^n)\}\right)\\
& <\frac{1}{2}\max\{1,h_F(E^n)\}+e^{256}\log(2+\max\{1,h_F(E^n)\}).
\end{aligned}
$$ 
 Since the Faltings height satisfies $h_F(A_1\times A_2)=h_F(A_1)+h_F(A_2)$ (cf. equation (2.7) in \cite{Bost} for instance) and we chose $n=80$, we find
$$
h_\Theta(E^n)<\frac{1}{2}\max\{1,80h_F(E)\}+e^{256}\log(2+\max\{1,80h_F(E)\}).
$$ 
Lemme 7.9 in \cite{GaudRem} (see also \cite{SilvermanE}) gives $h_F(E)\leq h(j_0)/12-0.72< h(j_0)$ where $h(x)=\log H(x)$ for $x\in \Q$ (note that the normalization of the Faltings height in \cite{GaudRem} and \cite{Pazuki} is the same), from which we get
$$
\begin{aligned}
h_\Theta(E^n)&< \frac{1}{2}\max\left\{1,\frac{20}{3}h(j_0)\right\}+e^{256}\log\left(2+\max\left\{1,\frac{20}{3}h(j_0)\right\}\right)\\
&\le\frac{10}{3}(1+h(j_0)) + e^{256}\log\left(\frac{20}{3}(1+h(j_0))\right)\\
&\le \max\left\{ \frac{20}{3}(1+h(j_0)) , e^{257}\log\left(\frac{20}{3}(1+h(j_0))\right) \right\}.
\end{aligned}
$$
From \eqref{EqnExpl1} we get
$$
\begin{aligned}
\log c(j_0,3)&< 34\log 2 + \max\left\{ \log \frac{20}{3}+ \log (1+h(j_0)) , 257 + \log\left( \log \frac{20}{3}+ \log(1+h(j_0))\right)\right\}\\
&< 24 +\max\{ 2 + \log (1+h(j_0)) , 257 + \log\left( 2 + \log(1+h(j_0))\right)\}\\
&= 26 + \max\{  \log (1+h(j_0)) , 255 + \log\left( 2 + \log(1+h(j_0))\right)\} .
\end{aligned}
$$
Finally \eqref{EqnExpl0} gives
$$
\log C(j_0,3)< 80^{46}\left(209 +  \max\{  \log (1+h(j_0)) , 255 + \log\left( 2 + \log(1+h(j_0))\right)\}\right)
$$
which proves \eqref{Eqncj0}. \qed

\section{Bounding the rank and applications}\label{SecApplications}

\subsection{Pointwise rank bounds}\label{SecPtRkBd} Let us recall the following bound for the rank of the quadratic twist of an elliptic curve. A similar result holds over number fields, although to simplify the notation we only state the case of $\Q$. Here we recall that $\omega(D)$ is the number of distinct prime divisors of $D$.
\begin{lemma}\label{LemmaRkBd} Let $E$ be an elliptic curve over $\Q$. There is an effectively computable constant $c(E)$ depending only on $E$ such that for all squarefree integers $D$ we have
$$
\rk E^{(D)}(\Q)\le 12\cdot \omega(D) + c(E).
$$
\end{lemma}
\begin{proof} This follows from \cite{SilvermanAEC2ndEd} Ch.\ VIII, Exercise 8.1 with $m=2$. In particular, $c(E)$ is bounded by a multiple of the number of places of bad reduction of $E$ plus  the $2$-rank of the class group of $K=\Q(E[2])=\Q(E^{(D)}[2])$. 
\end{proof}
From this we get
\begin{corollary} \label{CoroTwistGen} Let $E$ be an elliptic curve over $\Q$ and let $d$ be a positive integer. There is an effectively computable constant $C(E,d)$ depending only on $E$ and $d$ such that the following holds:

Let $D$ be a squarefree integer. Let $g$ be a rational function on $E^{(D)}$ defined over $\Q$ of degree $d$. If for some $N$ there is a sequence of rational points $P_1,...,P_N\in E^{(D)}(\Q)$ satisfying that $g(P_1),...,g(P_N)$ is a non-trivial arithmetic progression in $\Q$, then $N< C(E,d)^{\omega(D)+1}$.
\end{corollary}
\begin{proof} The result is obtained from Corollary \ref{Coro1} with $L=\Q$, using the fact that taking quadratic twists does not change the $j$-invariant, and using Lemma \ref{LemmaRkBd} to bound $\rk E^{(D)}(\Q)$.
\end{proof}
We note that Corollary \ref{CoroTwistIntro} is a special case of Corollary \ref{CoroTwistGen}.

\subsection{Average bounds for Mordell curves}\label{SecAvgMordell} As in the introduction, for $x>0$ we let $S(x)$ be the set of sixth-power free integers $n$ with $|n|\le x$. It is an elementary result in Analytic Number Theory that the number of $k$-power free positive integers up to $x$ is asymptotic to $x/\zeta(k)$ where $\zeta(s)$ is the Riemann zeta function. In particular we have 
\begin{lemma} \label{LemmaZeta} As $x\to \infty$ we have the asymptotic estimate
$$
\#S(x) \sim \frac{2}{\zeta(6)}\cdot x = \frac{1890}{\pi^6}\cdot x.
$$
\end{lemma}
For $n$ a sixth-power free integer, we consider the Mordell elliptic curve $A_n$ defined by the equation $y^2=x^3+n$. The following theorem is a special case of a result due to Fouvry, cf.\ \cite{Fouvry} Th\'eor\`eme 1 (using the bounds $R^+(\sqrt{3})\le 115$ and $R^-(\sqrt{3})\le 100$ given there). 
\begin{theorem} \label{ThmFouvry} The following estimate holds for all large enough $x$:
$$
\sum_{n\in S(x)} 3^{(\rk A_n(\Q))/2} < 216\cdot x.
$$
\end{theorem}
With these results, we can proceed to the proof of Theorem \ref{ThmAvgMohanty}.
\begin{proof}[Proof of Theorem \ref{ThmAvgMohanty}] By Theorem \ref{ThmIntro} with $j_0=0$ we have
$$
 \max\{\beta_x(A_n),\beta_y(A_n)\} \le \exp\left((1+ \rk A_n(\Q))/c\right)
$$
where $c=c(0)>0$ is an absolute constant. Let us take  $\tau = (c\cdot \log 3)/2>0$ and note that
$$
\max\{\beta_x(A_n),\beta_y(A_n)\}^\tau \le \sqrt{3}\cdot 3^{(\rk A_n(\Q))/2}.
$$
Theorem \ref{ThmFouvry} gives that for all large enough $x$ we have
$$
\sum_{n\in S(x)} \max\{\beta_x(A_n),\beta_y(A_n)\}^\tau \le 216 \sqrt{3} \cdot x.
$$
Since $216 \sqrt{3}\cdot 1890/\pi^6<735.5$, Lemma \ref{LemmaZeta} gives that for all large enough $x$ we have
$$
\sum_{n\in S(x)} \max\{\beta_x(A_n),\beta_y(A_n)\}^\tau < 800\cdot \#S(x).
$$
The result follows
\end{proof}
We remark that we can use Corollary \ref{Coro1} instead of Theorem \ref{ThmIntro} to obtain a version of Theorem \ref{ThmAvgMohanty} for more general rational functions, not just $x$ and $y$-coordinates.

As explained in the introduction, a crucial aspect in the proof of Theorem \ref{ThmAvgMohanty} is that our lower bounds for the rank are logarithmic on the maximal length of an arithmetic progression (cf.\ Theorems \ref{ThmIntro} and \ref{ThmMainArith}), and not worse than logarithmic. Briefly, the reason why Fouvry's theorem can control averages of an exponential function of the rank that  the core of \cite{Fouvry} is an average estimate for the size of certain $3$-isogeny Selmer groups, which is then used as an upper bound for $\# \left(A_n(\Q)/3A_n(\Q)\right)\ge 3^{\rk A_n(\Q)}$.

\subsection{Average bounds for congruent number curves}\label{SecAvgCongruent} 

For $x>0$, let $Q(x)$ be the set of odd squarefree positive integers $n\le x$.  We remark that it is an elementary exercise in sieve theory to check that $\#Q(x)\sim (3/\pi^2)\cdot x$ as $x\to \infty$. 

Given a squarefree integer $n$, we consider the elliptic curve $B_n$ defined by $y^2=x^3-n^2x$. These elliptic curves are associated to the classical congruent number problem. The following theorem is a direct consequence of Theorem 1 in \cite{HB} by Heath-Brown.
\begin{theorem} \label{ThmHB} Let $\ell$ be a positive integer. As $x\to \infty$ we have the asymptotic estimate
$$
\sum_{n\in Q(x)} \left(\#S_2(B_n)\right)^\ell\sim 4^{\ell+1}\prod_{j=1}^\ell\left(1+2^j\right)\cdot \# Q(x).
$$
In particular, there is a positive constant $\gamma(\ell)$ depending only on $m$ such that for all $x>1$ we have
$$
\sum_{n\in Q(x)} \left(\#S_2(B_n)\right)^\ell< \gamma(\ell) \cdot \# Q(x).
$$
\end{theorem}
 The previous bound for the moments of $\#S_2(B_n)$ allows us to prove Theorem \ref{ThmAvgCongruent}.
\begin{proof} All the elliptic curves $B_n$ have $j$-invariant equal to $1728$. For each squarefree positive integer $n$, Theorem \ref{ThmIntro} with $j_0=1728$ gives
$$
 \max\{\beta_x(B_n),\beta_y(B_n)\} \le \exp\left((1+ \rk B_n(\Q))/c\right)
$$
where $c=c(1728)>0$ is an absolute constant. Given $k>0$ we choose the positive integer
$$
\ell=\ell(k)=\left\lceil\frac{k}{c\log 2}\right\rceil
$$
where $\lceil t\rceil$ is the smallest integer bigger than or equal to $t$. Thus, $\ell\log 2\ge k/c$ and we deduce
$$
 \max\{\beta_x(B_n),\beta_y(B_n)\}^k\le 2^{(1+ \rk B_n(\Q))\cdot \ell} < \#\left(B_n(\Q)/2B_n(\Q)\right)^\ell
$$
for each squarefree positive integer $n$. Here we used the classical fact that the rational torsion of $B_n$ is isomorphic to $\Z/2\Z\times \Z/2\Z$, so that  $ \#\left(B_n(\Q)/2B_n(\Q)\right) = 2^{2+\rk B_n(\Q)}$.

The fundamental injective map $B_n(\Q)/2B_n(\Q)\to S_2(B_n(\Q))$ (cf.\ the exact sequence \eqref{EqnExact}) together with Theorem \ref{ThmHB} finally give that for all $x>1$ 
$$
\sum_{n\in S(x)}  \max\{\beta_x(B_n),\beta_y(B_n)\}^k < \gamma\left(\left\lceil\frac{k}{c\log 2}\right\rceil\right)\cdot \#Q(x)
$$
\end{proof}
Finally, we remark that using Corollary \ref{Coro1} instead of Theorem \ref{ThmIntro}, the same argument gives a version of Theorem \ref{ThmAvgCongruent} for any non-constant rational function on $B_n$, not just $x$ and $y$ coordinates.


\section{Acknowledgments}

The first author was supported by the FONDECYT Iniciaci\'on en Investigaci\'on grant 11170192, and the CONICYT PAI grant 79170039. The second author was supported by FONDECYT Regular grant 1190442.

We thank Dino Lorenzini and Eduardo Friedman for encouraging us to extend our methods beyond the case of arithmetic progressions in $x$ and $y$-coordinates. Also, we are indebted to \'Eric Gaudron for valuable suggestions regarding effectivity.

%
%
%


\end{document}